\documentclass[12pt,a4paper]{amsart}
\usepackage{amsmath}
\usepackage{amsthm}
\usepackage{amsfonts}
\usepackage{amssymb}
\usepackage{dsfont}
\usepackage{xr}
\usepackage{hyperref}

\newlength{\abstand}
\DeclareFontFamily{OT1}{pzc}{}
\DeclareFontShape{OT1}{pzc}{m}{it}{<-> s * [1.150] pzcmi7t}{}
\DeclareMathAlphabet{\mathcal}{OT1}{pzc}{m}{it}

\def\F{{\mathds F}}
\def\Fpb{{\overline{{\mathds F}_p}}}

\def\C{{\mathds C}}
\def\N{{\mathds N}}

\def\Q{{\mathds Q}}
\def\Z{{\mathds Z}}
\def\P{{\mathds P}}

\def\Zs{\str{\Z}}

\def\Qb{\overline{\Q}}

\def\fX{{\mathfrak X}}
\def\fB{{\mathfrak B}}
\def\fU{{\mathfrak U}}

\def\cL{\mathcal L}

\def\cO{\mathcal O}

\def\et{{\mathrm{\acute{e}t}}}

\newcommand\spec[1]{{\mathrm{Spec}\,(#1)}}

\newcommand\str[1]{{\mbox{}^*#1}}

\swapnumbers
\theoremstyle{definition}
\newtheorem{defi}{Definition}[section]

\newtheorem{lemma}[defi]{Lemma}

\newtheorem{rem}[defi]{Remark}

\newtheorem{thm}[defi]{Theorem}
\newtheorem{cor}[defi]{Corollary}
\newtheorem{conj}[defi]{Conjecture}

\input xy
\xyoption{all}

\title{Uniform bounds and ultraproducts of cycles} 
\author{Lars Br\"unjes, Christian Serp\'e}
\date{\today}
\address{Lars Br\"unjes}
\email{lbrunjes@gmx.de}
\address{
  Christian Serp\'e \\
  Westf\"alische Wilhelms-Universit\"at M\"unster \\
  Mathematisches Institut \\
  Einsteinstrasse 62
  D-48149 M\"unster \\
  Germany}
\email{serpe@uni-muenster.de}


\date{\today}

\begin{document}

\begin{abstract}
This paper is about the question whether a cycle in the l-adic
cohomology of a smooth projective variety over $\Q$,
which is algebraic over almost all finite fields $\F_p$, is also algebraic over
$\Q$. We use ultraproducts respectively nonstandard techniques in the sense of A.
Robinson, which the authors applied systematically to algebraic geometry in \cite{enlsch}
and \cite{EoSII}. We give a reformulation of the question in form of
uniform bounds for the complexity of algebraic cycles over finite
fields.

\end{abstract}

\maketitle

\section{Introduction}

Let $X$ be a smooth and projective variety over $\Q$, and let $l$ be
a prime. Let $\eta\in H^{2i}(X_{\Qb},\Q_l(i))$ be a class in the
geometric l-adic cohomology. This paper concerns the question whether $\eta$
is \emph{algebraic over $\Q$}, i.e. whether there is an
$\alpha \in CH^i(X)\otimes_{\Z}\Q_l$ such that
$cl_{\Q}(\alpha)=\eta$, where
$$cl_{\Q}: CH^i(X)\otimes \Q_l\rightarrow H^{2i}(X_{\Qb},\Q_l(i))$$
is the cycle class map into l-adic cohomology.

\vspace{\abstand}

Over an
open dense part $\fB$  of $\spec \Z$ the variety $X$ has a smooth and
projective model $\fX$. Therefore by smooth and proper base change in \'{e}tale cohomology, for almost all primes $p$ we can identify
$$H^{2i}_\et(X_{\Qb},\Q_l(i))\cong H^{2i}_\et(\fX_{\Fpb}, \Q_l(i)).$$

Further there is a specialization map for Chow groups
$$CH^i(\fX_{\Q})\rightarrow CH^i(\fX_{\F_p})$$
such that the diagram
\begin{displaymath}\label{comdiafund}
\xymatrix{
CH^i(\fX_{\Q})\otimes_{\Z}\Q_l \ar[r] \ar[d]_{cl_{\Q}} & CH^i(\fX_{\F_p})\otimes_{\Z}\Q_l \ar[d]^{cl_{\F_p}} \\
H^{2i}_\et(\fX_{\Qb},\Q_l(i))\ar[r]^{\sim} & H^{2i}_\et(\fX_{\Fpb},
\Q_l(i))
}
\end{displaymath}

commutes. Therefore we know that if $\eta$ is algebraic over $\Q$,
then for almost all primes $p$ there is an
$\alpha_p\in CH^i(\fX_{\F_p})\otimes \Q_l$ with
$cl_{\F_p}(\alpha_p)=\eta$, i.e. $\eta\in
H^{2i}_\et(\fX_{\overline{\F_p}}, \Q_l(i))$
is algebraic over $\F_p$. It is a conjecture
that the converse is also true:

\vspace{\abstand}

\begin{conj}\label{dreamtheorem}
In the above situation $\eta\in H^{2i}_\et(X_{\Qb},\Q_l(i))$ is
algebraic over $\Q$ if and only if it is algebraic over $\F_p$ for almost all primes $p$.
\end{conj}

This conjecture would for example follow from the Tate conjecture for $X$.

\vspace{\abstand}

If $\eta$ is indeed algebraic over $\Q$, we know that all the $\alpha_p\in
CH^i(\fX_{\F_p})\otimes \Q_l$ are specializations of one single $\alpha\in CH^i(X)\otimes \Q_l$.
And this tells us something about the complexity of the $\alpha_p$. To make
this more precise, we define:
\begin{defi}\label{firstcyclescomplexity}
Let $X$ be a projective scheme over a field $k$  with a fixed
projective embedding $X\hookrightarrow \P^n_k$, and let $\alpha\in CH^i(X)\otimes\Q_l$.
We say that \emph{the complexity of $\alpha$ is bounded by $c\in\N$}
if we can write $\alpha$ as $\alpha=\sum_{i=1}^{r} a_i [Z_i]$
with $a_i\in\Q_l$ and $Z_j\hookrightarrow X$ such that the following holds:
$$ r<c\text{, and for all } i=1,\dots,r \text{ we have }
|a_i|_l<c \text{ and } \deg(Z_i)<c.$$
\end{defi}

\vspace{\abstand}

Now the complexity of every $\alpha\in CH^i(X)\otimes \Q_l$ is obviously bounded by
\emph{some} $c_0\in \N$. It can be seen easily that the complexity of the
resulting family of specializations ${\alpha_p}$ is then uniformly
bounded by $c_0$ as well. The main result of this paper is the converse of this
observation:

\vspace{\abstand}

\begin{thm}[cf. corollary \ref{cor:main}]
In the above situation let $\eta\in H_\et^{2i}(X_{\Qb},\Q_l(i))$. We
assume that there is a $c_0\in\N$ such that for almost all primes $p$ there is
an $\alpha_p\in CH^i(\fX_{\F_p})\otimes\Q_l$ with complexity bounded by
$c_0$ and with $cl_{\F_p}(\alpha_p)=\eta$. Then there is an $\alpha\in
CH^i(X)\otimes \Q_l$ with $cl_{\Q}(\alpha)=\eta$, i.e. $\eta$ is
algebraic over $\Q$
\end{thm}

\vspace{\abstand}

In particular this means that in order to decide the question whether a cohomology class is
algebraic over $\Q$, it is enough to understand the case of all finite fields uniformly well enough.

\vspace{\abstand}

In the integral case, where we consider the image of the cycle
class map
$$CH^i(X)\rightarrow H^{2i}_\et(X_{\Qb},\Z_l),$$
we prove a stronger version for $i=1$, which does not need the notion of complexity
(cf. theorem \ref{thm:withoutcompl}):

\begin{thm}
In the above situation let $\eta\in H_\et^{2}(X_{\Qb},\Z_l(1))$.
We assume that for almost all primes $p$ there is a
$z_p\in CH^1(\fX\otimes_{\Z}\F_p)$ such that
$$cl_{\Fpb}(z_p)=\eta.$$
Then there is a field extension $K/\Q$ and an element $z\in
CH^1(\fX_K)$ with
$$cl_{K}(z)=\eta \text{ in } H^2_\et(\fX_{\overline{K}},\Z_l(1)).$$
\end{thm}

\vspace{\abstand}

The basic idea for the proofs is the following: By our assumption there is a cofinite
set $S\subset\P$ of the set of primes $\P$ such that for all $p\in S$ there are cycles
$\alpha_p\in CH^i(\fX_{\F_p})\otimes\Q_l$ with $cl_{\F_p}(\alpha_p)=\eta$.
Now we choose an ultrafilter $\fU$ on $S$ and consider the ultraproduct of these cycles
$$(\alpha_p)_{p\in S,\fU}\in\prod_{p\in S, \fU} (CH^i(\fX_{\F_p})\otimes \Q_l).$$
In our papers  \cite{EoSII} we studied how Chow groups and \'{e}tale cohomology behave under
such ultraproducts. In particular  in loc. cit. we constructed morphisms
\begin{equation}\label{firstN}
CH^i\left(\fX\otimes_{\Z}\prod_{p\in S,\fU} \F_p\right) \rightarrow
   \prod_{S,\fU}\left(CH^i(\fX_{\F_p})\right)
\end{equation}
and
\begin{equation}\label{secondN}
H_\et^{2i}\left(\fX\otimes_{\Z}\prod_{p\in S,\fU} \overline{\F_p},\Z/l^n\Z(i)\right)
\rightarrow \prod_{S,\fU} H_\et^{2i}\left(\fX_{\overline{\F_p}},\Z/l^n\Z(i)\right)
\end{equation}
and described the image of these morphisms: The image
of (\ref{firstN}) can be described by the complexity
of the occurring cycles, whereas morphism (\ref{secondN})
is an isomorphism in our situation.
Furthermore, for all $n\in\N$ we have a commutative diagram
$$
\xymatrix@C=50mm{
\prod_{S,\fU}\left(CH^i(\fX_{\F_p})\right) \ar[r]^{\prod cl} & \prod_{S,\fU} H_\et^{2i}\left(\fX_{\overline{\F_p}},\Z/l^n\Z(i)\right) \\
CH^i\left(\fX\otimes_{\Z}\prod_{p\in S,\fU} \F_p\right) \ar[r]^{cl} \ar[u] &
   H_\et^{2i}\left(\fX\otimes_{\Z}\prod_{p\in S,\fU} \overline{\F_p},\Z/l^n\Z(i)\right). \ar[u]
}
$$

Now $\prod_{p\in S,\fU}\F_p$ is a field of characteristic zero,
so an extension of $\Q$, and the assumptions imply that our cycle
$\eta\in H^{2i}_{et}(X_{\Qb},\Z_l(i))$ lies in the image of
$$cl:CH^i\left(\fX\otimes_{\Z}\prod_{p\in S,\fU} \F_p\right)\otimes_{\Z}\Q_l \rightarrow
   H_\et^{2i}\left(\fX\otimes_{\Z}\prod_{p\in S,\fU} \overline{\F_p},\Z_l(i)\right).$$

\vspace{\abstand}

Instead of using ultraproducts directly, we use the more general method of enlargements in the sense of A.~Robinson,
which the authors applied to modern algebraic geometry in \cite{enlsch} and \cite{EoSII}.

\vspace{\abstand}

All results of this paper remain true (with minor modifications) if we
start with a smooth projective variety over an arbitrary number field.
The proofs need only little modifications in this case.
Yet we decided to stick to the case of a smooth projective variety over $\Q$ to make notations easier and the paper more readable.

\vspace{\abstand}

In section 2 we give the basic definitions we need and introduce an appropriate notion of complexity of cycles.

In section 3 we give different versions of our main result and discuss the integral case for divisors without using the notion of complexity.

\section{Basic notations}

First we fix some notations. For a smooth, proper and connected scheme
$X$ over a field $k$, we denote by $Z_i(X)$ the free abelian group generated by those integral
subschemes $Y\hookrightarrow X$ of $X$ with $\dim(Y)=i$ . By $Z_i^{eff}(X)$ we
denote the submonoid of $Z_i(X)$ of \emph{effective cycles}, i.e. elements of the form $\sum_j \alpha_j [Z_j]$
with $\alpha_i\in \N$. By $CH_i(X)$ we denote the quotient of
$Z_i(X)$ by rational equivalence, and by $CH^{eff}_i(X)$ we denote the
submonoid of $CH_i(X)$ generated by elements of $Z^{eff}_i(X)$.
We also want to use the grading by codimension, and so we denote
$$ Z^i(X):=Z_{\dim(X)-i}(X),$$
$$ Z_{eff}^i(X):=Z^{eff}_{\dim(X)-i}(X),$$
$$ CH^i(X):=CH_{\dim(X)-i}(X),$$
$$ CH^i_{eff}(X):=CH^{eff}_{\dim(X)-i}(X).$$

\vspace{\abstand}

For 0-cycles we denote by
$$\deg:CH_0(X)=CH^{\dim(X)}(X) \rightarrow \Z$$
the \emph{degree} map which is defined by
$\sum_j \alpha_j[Z_j]\mapsto \sum_j \alpha_j [\kappa(Z_j):k].$
For higher dimensional cycles we first fix an ample line bundle $\cL$
on $X$
and denote by $c_1(\cL)$ the first Chern class of $\cL$. With that we
define the degree in general as
$$
\begin{array}{cccc}
\deg_{\cL}: & CH_i(X)=CH^{\dim(X)-i}(X) & \rightarrow & \Z \\
           &  x                & \mapsto     & \deg(c_1(\cL)^i\cap x).
\end{array}
$$

\vspace{\abstand}

\begin{rem}
If $X\hookrightarrow \P^n$ is a closed embedding and $\cL:=\cO_{\P^n}(1)_{\vert X}$,
then the above notion gives the ordinary degree of integral subschemes as subschemes
of $\P^n$. For details we refer to \cite{fulton}[section 2.5].
\end{rem}

\vspace{\abstand}

Next consider the cycle class map
$$cl_{k}: CH^i(X)\rightarrow H^{2i}_\et(X_{\overline{k}},\Z_l(i)),$$
where the right hand side is the l-adic cohomology
$$H^{2i}_\et(X_{\overline{k}},\Z_l(i)):=
\varprojlim_{n} H^{2i}_\et(X_{\overline{k}},\mu_{l^n}^{\otimes i}).$$

\vspace{\abstand}

\begin{lemma}
There is a degree map
$$\deg_{\cL}^{\et}:H_\et^{2i}(X,\Z_l)\rightarrow\Z_l$$
such that we have $\deg_{\cL}^{\et}\circ cl=\deg_{\cL}$.
\end{lemma}
\begin{proof}
The follows from the fact that the cycle map is compatible with
products.
\end{proof}

\vspace{\abstand}

\begin{defi}\label{cyclescomplexity}
Let $k$ be a field, $X$ a proper, smooth, connected scheme over $k$ and
$c\in\N$ a natural number.
\begin{enumerate}
\item We say that an element $x\in CH^i(X)$ has \emph{complexity less than
$c$} and write $compl(x)<c$ iff we can write $x$ as
a difference $x=x_0-x_1$ of two effective  cycles $x_0,x_1\in CH^i(X)$
with
$$\deg(x_0)<c \text{ and } \deg(x_1)<c.$$
\item We say that an element $x\in CH^i(X)\otimes\Q$ has \emph{complexity less than
$c\in\N$} and write $compl(x)<c$ iff there is an $x'\in
CH^i(X)$ with
$compl(x')<c$ and an $\alpha\in\N$ with $\alpha<c$ such that  $\alpha \cdot x=x'$.
\item We say that an element $x\in CH^i(X)\otimes \Z_l$ has \emph{complexity less than
$c\in\N$} and write $compl(x)<c$ iff we can write it as
$x=\sum_{i=1}^{n}\alpha_i x_i$ with $\alpha_i\in\Z_l$ and $x_i\in
CH^i(X)$ such that for all $i=1,\dots, n$ we have
$$compl(x_i)<\frac{c}{n}.$$
\item We say that an element $x\in CH^i(X)\otimes\Q_l$ has \emph{complexity less than
$c\in\N$} and write $compl(x)<c$ iff there is an $x'\in
CH^i(X)\otimes \Z_l$ with
$compl(x')<c$ and an $\alpha\in\Q_l$ with $|\alpha|_l<c$ such that  $\alpha \cdot x'=x$.
\end{enumerate}
\end{defi}

\vspace{\abstand}

\begin{rem}
The reason why (ii) and (iv) seems to be inverse to each other is the following:
$\Z\subset \Q$ is discrete and unbounded in the archimedian metric and $\Z_l\subset \Q_l$
is not discrete and bounded in the $l$-adic topology.
\end{rem}

\vspace{\abstand}

\begin{rem}
Part (iv) of definition \ref{cyclescomplexity} gives essentially the same
notion of complexity as definition \ref{firstcyclescomplexity} from the
introduction. For more details we refer to remark \ref{comparecomplexitydef}.
\end{rem}

\vspace{\abstand}

\begin{rem}
It would be nice to be able to give the complexity only in terms of
the degree of cycles, but that is in
general not possible. To see why, consider a smooth, projective surface $F$ over $\C$ with
$p_g\neq 0$. By \cite{mumfordfinite}, there is no $n\in\N$ such that the map
\begin{equation}\label{mumfordsym}
Sym^n F\times Sym^n F \rightarrow CH_0(X)_0
\end{equation}
is surjective. Here $CH_0(X)_0$ denotes $0$-cycles of degree zero, and $Sym^n$ denotes the n-th
symmetric power. Furthermore, $x\in CH_0(X)_0$ is a cycle with $compl(x)<n$ if and only if $x$ lies
in the image of \eqref{mumfordsym}.
\end{rem}

\vspace{\abstand}

\begin{rem}
The behavior of this notion of complexity under intersection products
is studied in \cite{EoSII}[theorem 5.11].
\end{rem}

\vspace{\abstand}

By $*:\hat{S}\rightarrow \hat{W}=\hat{\str{S}}$ we always denote an
enlargement of the superstructure $\hat{S}$ in the sense of nonstandard
analysis (cf. for example \cite{loeb}[section I.2]). We rigorously use
the concept of enlarging categories, schemes, cycles and  \'etale cohomology,
which the authors developed in \cite{enlcat}, \cite{enlsch} and \cite{EoSII}.
We always assume that our fixed superstructure $\hat{S}$ is large enough,
so that all needed categories are $\hat{S}$-small in the sense of \cite{enlcat}.

\begin{rem}\label{comparecomplexitydef}
\mbox{}\\[-5mm]
\begin{enumerate}
\item  Let $K$ be an internal field, $X$ a *proper, *smooth, *connected *scheme with a *ample
$\cO_X$-*module $\cL$. Then by transfer we get the degree map
$$\str{\deg_{\cL}}:\str{CH}^i(X) \rightarrow \str{\Z}$$
and a corresponding notion of *complexity. Then the complexity of a *cycle on $X$ can be bound by an
element of $\N\subset\str{\N}$ in the sense of definition \ref{firstcyclescomplexity}
if and only if it can be bound by an element of $\N\subset\str{\N}$ in the sense of
definition \ref{cyclescomplexity}.
\item If we compare part (i) of definition \ref{cyclescomplexity} with
definition \cite{EoSII}[5.1] the same as in (i) is true.
\end{enumerate}
\end{rem}

If $\fX\rightarrow \fB$ and $\spec A\rightarrow \fB$ are morphisms of schemes,
we denote the fibre product by $\fX_A:=\fX \times_{\fB} \spec A$. Also,
if $\fX\rightarrow \fB$ and $\spec A\rightarrow \fB$ are *morphisms of *schemes,
we denote the *fibre product by $\fX_A:=\fX \str{\times_{\fB}} \spec A$.
Finally, for a field (respectively internal field) $k$,
we denote a separable closure (respectively *separable *closure) of $k$ by $\overline{k}$.

\section{Lifting cycles modulo homological equivalence}

In this section we consider mainly the following situation: Let
$\fB\subset \spec \Z$ be an open, non-empty subset, $\fX\rightarrow \fB$ a
smooth and proper morphism and $\cL$ an ample sheaf on $\fX$. For a
point $s\in \fB$ we
consider the l-adic cohomology groups of the geometric fibre
$$ H^{2i}_\et(\fX_{\overline{\kappa(s)}},\Z_l(i)):=
\varprojlim_n H^{2i}_\et(\fX_{\overline{\kappa(s)}},\mu_{l^n}^{\otimes i}),$$
$$ H^{2i}_\et(\fX_{\overline{\kappa(s)}},\Q_l(i)):=
H^{2i}_\et(\fX_{\overline{\kappa(s)}},\Z_l(i))\otimes_{\Z_l} \Q_l.$$
By the smooth and proper base change theorems we know that the
specialization homomorphisms
$$H^{2i}_\et(\fX_{\overline{\kappa(s)}},\mu_{l^n}^{\otimes i})\rightarrow
     H^{2i}_\et(\fX_{\overline{\Q}},\mu_{l^n}^{\otimes i}),$$
$$H^{2i}_\et(\fX_{\overline{\kappa(s)}},\Z_l(i))\rightarrow
     H^{2i}_\et(\fX_{\overline{\Q}},\Z_l(i)),$$
and
$$H^{2i}_\et(\fX_{\overline{\kappa(s)}},\Q_l(i))\rightarrow
     H^{2i}_\et(\fX_{\overline{\Q}},\Q_l(i))$$
are isomorphisms.
Also, if $K/\Q$ is a field extension with a chosen embedding
$\overline{\Q}\rightarrow\overline{K}$, the pullback homomorphisms
$$H^{2i}_\et(\fX_{\overline{\Q}},\mu_{l^n}^{\otimes i})
\rightarrow H^{2i}_\et(\fX_{\overline{K}},\mu_{l^n}^{\otimes i}),$$
$$H^{2i}_\et(\fX_{\overline{\Q}},\Z_l(i))
\rightarrow H^{2i}_\et(\fX_{\overline{K}},\Z_l(i)),$$
and
$$H^{2i}_\et(\fX_{\overline{\Q}},\Q_l(i))
\rightarrow H^{2i}_\et(\fX_{\overline{K}},\Q_l(i))$$
are isomorphisms. In what follows, we identify these groups without
further mentioning.

\vspace{\abstand}

If $K$ is a field and $\spec{K}\rightarrow \fB$ is a morphism, we consider the
following cycle class maps
\begin{enumerate}
\item $cl_K: CH^i(\fX_K)\rightarrow H_\et^{2i}(\fX_{\overline{K}},\Z_l(i)),$
\item $cl_K: CH^i(\fX_K)\otimes \Q \rightarrow H_\et^{2i}(\fX_{\overline{K}},\Q_l(i)).$
\item $cl_K: CH^i(\fX_K)\otimes \Z_l \rightarrow H_\et^{2i}(\fX_{\overline{K}},\Z_l(i)),$
\item $cl_K: CH^i(\fX_K)\otimes \Q_l \rightarrow H_\et^{2i}(\fX_{\overline{K}},\Q_l(i)).$
\end{enumerate}
In each case we investigate the question whether a cohomology class in the l-adic cohomology,
which is algebraic over the finite fields in $\fB$, is also algebraic over a field of characteristic zero.
For that we use our notion of complexity of section 2.

\vspace{\abstand}
{\bf (i) The case
$cl_K: CH^i(\fX_K)\rightarrow H_\et^{2i}(\fX_{\overline{K}},\Z_l(i))$}
\vspace{\abstand}

\begin{thm}\label{firstmaintheorem}
Let as above $\fX\rightarrow \fB$ be a smooth, proper morphism, $\fB\subset
\spec \Z$ open and non empty, $\cL$ an ample sheaf on $\fX$ and
$\eta\in H^{2i}_\et(\fX_{\overline{\Q}},\Z_l(i))$ a cohomology class of
the geometric generic fibre of $\fX$. We assume that there is a
constant $c\in\N$ such that there are infinitely many closed points
$b\in \fB$ with:
\begin{quote}
There is a $z_b\in CH^i(\fX_{\kappa(b)})$ with $cl(z_b)=\eta$ and
$compl(z_b)<c$ with respect to the ample sheaf
$\cL|_{\fX_{\kappa(b)}}$.
\end{quote}
Then there is a field extension $K/\Q$ and an element $z_K\in
CH^i(\fX_K)$ with $cl(z_K)=\eta$ in $H^{2i}_\et(\fX_{\overline{K}},\Z_l(i))$.
\end{thm}

For the proof we use the following lemma:

\begin{lemma}\label{lem:imageN}
Let $K$ be an internal field, X a proper scheme over $K$ and $\cL$ an
ample sheaf on $X$. Then the image of the morphism
$$N:CH^i(X)\rightarrow \str{CH}^i(N(X))$$
consists exactly of those elements in $\str{CH}^i(N(X))$ whose
*complexity is less than $c$ for a natural number $c\in\N$.
\end{lemma}

\begin{proof}
This follows from lemma 5.2 in \cite{EoSII} and remark \ref{comparecomplexitydef} (ii).
\end{proof}

\begin{proof}[Proof of theorem \ref{firstmaintheorem}]
By transfer there are an $b\in \str{\fB}$, corresponding to an infinite
prime $P\in\str{\P}$, and $x_b\in\str{CH}^i(\str{\fX}_{\kappa(b)})$
with
\begin{equation}\label{firstone}
\str{compl}(x_b)<c
\end{equation}

and

\begin{equation}\label{secondone}
\str{cl}(x_b)=\str{\eta}.
\end{equation}

The field $\kappa(b)$ is externally of characteristic zero.
By definition (cf.\cite{enlsch}) we have $N(\fX_{\kappa(b)})=\str{\fX}_{\kappa(b)}$.
Therefore condition (\ref{firstone}) and lemma
\ref{lem:imageN} imply that $x_b$ lies in the image of
$$N:CH^i(\fX_{\kappa(b)})\rightarrow \str{CH}^i(\str{\fX}_{\kappa(b)}),$$
i.e. there is an $z_{\kappa(b)}\in CH^i(\fX_{\kappa(b)})$ with $N(z_{\kappa(b)})=x_b$.
For all $n\in\N$ the above morphism fits into the following commutative diagram
\begin{displaymath}\label{diagram}
\xymatrix@C=20mm{
CH^i(\fX_{\kappa(b)}) \ar[r]^N \ar[d] &
              \str{CH}^i(\str{\fX}_{\kappa(b)}) \ar[d] \\
H^{2i}_\et(\fX_{\overline{\kappa(b)}},\mu_{l^n}^{\otimes i})
\ar[r]^{N}_{\sim} &
\str{H}^{2i}_\et(\str{\fX}_{\str{\overline{\kappa(b)}}},\str{\mu}_{l^n}^{\otimes
  i}) \ar[d]^{\wr} \\
H^{2i}_\et(\fX_{\overline{\Q}},\mu_{l^n}^{\otimes i})\ar[u]^{\wr}
\ar[r]^{*}_{\sim} &
\str{H}^{2i}_\et(\str{\fX}_{\str{\overline{\Q}}},\str{\mu}_{l^n}^{\otimes i}).
}
\end{displaymath}

The upper square commutes by \cite{EoSII}, the lower square by the following lemma
\ref{commutative}, and the horizontal morphism in the middle
$$N: H^{2i}_\et(\fX_{\overline{\kappa(b)}},\mu_{l^n}^{\otimes i})
\rightarrow
\str{H}^{2i}_\et(\str{\fX}_{\str{\overline{\kappa(b)}}},\str{\mu}_{l^n}^{\otimes
  i})$$
is an isomorphism by \cite{EoSII}[Cor.2.14].

Thus we see that for all $n\in\N$ we have
$$cl(z_{\kappa(b)})=\eta \text{ in } H^{2i}_\et(\fX_{\kappa(b)},\mu_{l^{\otimes n}}^{\otimes i})$$
and therefore
$$ cl(z_{\kappa(b)})=\eta \text{ in } H^{2i}_\et(\fX_{\kappa(b)},\Z_l(i)).$$
\end{proof}

\begin{lemma}\label{commutative}
In the above situation we have the following commutative diagram:
\[
\xymatrix@C=30mm{
H^{2i}_\et(\fX_{\overline{\kappa(b)}},\mu_{l^n}^{\otimes i})
\ar[r]^{N}_{\sim} &
\str{H}^{2i}_\et(\str{\fX}_{\str{\overline{\kappa(b)}}},\str{\mu}_{l^n}^{\otimes
  i}) \ar[d]^{\wr} \\
H^{2i}_\et(\fX_{\overline{\Q}},\mu_{l^n}^{\otimes i})\ar[u]^{\wr}
\ar[r]^{*}_{\sim} &
\str{H}^{2i}_\et(\str{\fX}_{\str{\overline{\Q}}},\str{\mu}_{l^n}^{\otimes i}).
}
\]
\end{lemma}

\begin{proof}
Consider the base extension
$$\fX\otimes \str{\Z} \rightarrow \spec{\str{\Z}}.$$
The left arrow of the diagram can be identified with the inverse of the
base change homomorphism of the two points $(0)\in\str{\Z}$ and
$b\in\spec{\str{\Z}}$ (note that internal prime ideals are in particular
external ones). The lemma then follows from the compatibility of $N$
for \'{e}tale cohomology with the specialization homomorphism
(cf. \cite{EoSII}).
\end{proof}

\vspace{\abstand}

{\bf (ii) The case $cl_K: CH^i(\fX_K)\otimes \Q \rightarrow H_\et^{2i}(\fX_{\overline{K}},\Q_l(i))$}

\vspace{\abstand}

\begin{thm}\label{satz:einhauptsatz}
Let as above $\fX\rightarrow \fB$ be a smooth and proper morphism, $\fB\subset
\spec \Z$ open and non-empty, $\cL$ an ample sheaf on $\fX$ and
$\eta\in H^{2i}_\et(\fX_{\overline{\Q}},\Q_l(i))$ a cohomology class of
the geometric generic fibre of $\fX$. We assume that there is a
constant $c\in\N$ such that there are infinitely many closed points
$b\in \fB$ with:
\begin{quote}
There is a $z_b\in CH^i(\fX_{\kappa(b)})\otimes \Q$ with $cl(z_b)=\eta$ and
$compl(z_b)<c$ with respect to the ample sheaf
$\cL|_{\fX_{\kappa(b)}}$.
\end{quote}
Then there are a field extension $K/\Q$ and an element $z_K\in
CH^i(\fX_K)\otimes \Q$ with $cl(z_K)=\eta$ in $H^{2i}_\et(\fX_K,\Q_l(i))$.
\end{thm}

\begin{proof}
By transfer there are a $b\in \str{\fB}$, corresponding to an infinite
prime $P\in\str{\P}$, and $x_b\in\str{CH}^i(\str{\fX}_{\kappa(b)})\otimes \str{\Q}$
with $\str{compl}(x_b)<c$ and $\str{cl}(x_b)=\str{\eta}$.

By the definition of complexity there are an $\alpha\in\str{\N}$
with $\alpha<c$, in particular $\alpha\in\N$,
and $x'_b\in\str{CH}^i(\str{\fX_{\kappa(b)}})$ with $\str{compl}(x'_b)<c$ such
that $\alpha \cdot x_b= x_b'$.
As in the proof of theorem~\ref{firstmaintheorem} we see that there is a
$z_b \in CH^i(\fX_{\kappa(b)})$ with $$
N(z_b)=x'_b$$
and that
$$cl_{\kappa(b)}(z_b)=\alpha \cdot \eta,$$
and therefore we have
$$cl_{\kappa(b)}(\frac{1}{\alpha}\cdot z_b)=\eta$$
as desired.
\end{proof}

\vspace{\abstand}

\begin{rem}
One prominent example, where one would like to know whether a cohomology class lies in the image of
$$CH^i(X)\otimes \Q \rightarrow H^{2i}_\et(\overline{X},\Q_l(i)),$$
 is the case of the K\"unneth components. Over finite fields it is known by \cite{katzmessing}
that the K\"unneth components are algebraic. Unfortunately one uses the Frobenius morphism to construct
the cycles, and it is not possible to use this representation to find a uniform bound for the
complexity.
\end{rem}

\vspace{\abstand}

{\bf (iii) The case $cl_K: CH^i(\fX_K)\otimes \Z_l \rightarrow H_\et^{2i}(\fX_{\overline{K}},\Z_l(i))$}

\vspace{\abstand}

\begin{thm}
Let as above $\fX\rightarrow \fB$ be a smooth proper morphism, $\fB\subset
\spec \Z$ open and non-empty, $\cL$ an ample sheaf on $\fX$ and
$\eta\in H^{2i}_\et(\fX_{\overline{\Q}},\Z_l(i))$ a cohomology class of
the geometric generic fibre of $\fX$. We assume that there is a
constant $c\in\N$ such that there are infinitely many closed points
$b\in \fB$ with:
\begin{quote}
There is an $z_b\in CH^i(\fX_{\kappa(b)})\otimes \Z_l $ with $cl(z_b)=\eta$ and
$compl(z_b)<c$ with respect to the ample sheaf
$\cL|_{\fX_{\kappa(b)}}$.
\end{quote}
Then there are a field extension $K/\Q$ and an element $z_K\in
CH^i(\fX_K)\otimes \Z_l$ with $cl(z_K)=\eta$ in $H^{2i}_\et(\fX_K,\Z_l(i))$.
\end{thm}

\begin{proof}
By transfer there are a $b\in \str{\fB}$ corresponding to an infinite
prime $P\in\str{\P}$ and $x_b\in\str{CH}^i(\str{\fX}_{\kappa(b)})\otimes \str{\Z_l}$
with
\begin{equation}
\str{compl}(x_b)<c
\end{equation}

and

\begin{equation}
\str{cl}(x_b)=\str{\eta}.
\end{equation}.

Now by the definition of the complexity of elements in $\str{CH}^i(\str{\fX}_{\kappa(b)})\otimes \Z_l$,
we can write $x_b=\str{\sum}_{i=1}^n \alpha_i x_i$ with $\alpha_i\in\Zs_l$ and
$x_i\in \str{CH}^i(\fX_{\kappa(b)})$ with $\str{compl}(x_i)<\frac{c}{n}$.
A priori we have $n\in\str{\N}$, but if we assume $n\in\str{\N}-\N$ then $\frac{c}{n}$ would be infinitesimal
and $compl(x_i)<\frac{c}{n}$ together with $x_i\in\str{CH}^i(\str{\fX}_{\kappa(b)})$ would imply $x_i=0$. So
we assume that $n\in\N$ and that the sum $x_b=\sum_{i=1}^n \alpha_i x_i$ is finite.
By lemma \ref{lem:imageN}, as in the proof of theorem \ref{firstmaintheorem},
there are $z_{i,b}\in CH^i(\fX_{\kappa(b)})$ such that
$cl(z_{i,b})=x_i$. Now $\Q_l$ is
complete with respect to the $l$-adic norm, and $\Z_l=\bigl\{\alpha\in\Q_l\bigl\vert|\alpha|\leq1\bigr\}$. Therefore we have
the standard part map
$$st:\str{\Z_l}\rightarrow \Z_l$$
with the property that for all $\alpha\in\str{\Z_l}$, the difference $\alpha-st(\alpha)$ is infinitesimally small.
In particular this means that for all standard $n\in\N\subset\str{\N}$
\begin{equation}\label{standarteilabsch}
\alpha=st(\alpha) \text{ in } \str{\Z_l}/l^n\str{\Z_l}=\Z/l^n\Z.
\end{equation}

We can thus define
$$z_b:=\sum_{i=1}^n st(\alpha_i)\cdot z_{i,b},$$
use \eqref{standarteilabsch} and argue as in the proof of theorem \ref{firstmaintheorem} to show that
$$ cl_{\kappa(b)}(z_b)=\eta.$$
\end{proof}

\vspace{\abstand}

{\bf (iv) The case $cl_K: CH^i(\fX_K)\otimes \Q_l \rightarrow H_\et^{2i}(\fX_{\overline{K}},\Q_l(i))$}

\vspace{\abstand}

\begin{thm}\label{satz:einhauptsatz2}
Let as above $\fX\rightarrow \fB$ be a smooth and proper morphism, $\fB\subset
\spec \Z$ open and non-empty, $\cL$ an ample sheaf on $\fX$ and
$\eta\in H^{2i}_\et(\fX_{\overline{\Q}},\Q_l(i))$ a cohomology class of
the geometric generic fibre of $\fX$. We assume that there is a
constant $c\in\N$ such that there are infinitely many closed points
$b\in \fB$ with:
\begin{quote}
There is a $z_b\in CH^i(\fX_{\kappa(b)})\otimes \Q_l$ with $cl(z_b)=\eta$ and
$compl(z_b)<c$ with respect to the ample sheaf
$\cL|_{\fX_{\kappa(b)}}$.
\end{quote}
Then there are a field extension $K/\Q$ and an element $z_K\in
CH^i(\fX_K)\otimes \Q_l$ with $cl(z_K)=\eta$ in $H^{2i}_\et(\fX_K,\Q_l(i))$.
\end{thm}

\begin{proof}
By transfer there are an $b\in \str{\fB}$, corresponding to an infinite
prime $P\in\str{\P}$, and $x_b\in\str{CH}^i(\str{\fX}_{\kappa(b)})\otimes \str{\Q_l}$
with $\str{compl}(x_b)<c$ and $\str{cl}(x_b)=\str{\eta}$.

By the definition of complexity there are an $\alpha\in\str{\Q_l}$ and
$x'_b\in\str{CH}^i(\str{\fX_{\kappa(b)}})\otimes \str{\Z_l}$ such
that $|\alpha|_l<c$ and $x_b=\alpha_b\cdot x_b'$.
The claim then follows as in the proof of the previous theorem.
\end{proof}

\vspace{\abstand}

\begin{cor}\label{cor:main}
Let $\fX$, $\fB$ and $\cL$ be as in the theorem. Now we assume that there is a
constant $c\in\N$, such that for \emph{almost all} closed points $s\in\fB$ we have:
\begin{quote}
There is a $z_b\in CH^i(\fX_{\kappa(b)})\otimes \Q_l$ with $cl(z_b)=\eta$ and
$compl(z_b)<c$ with respect to the ample sheaf
$\cL|_{\fX_{\kappa(b)}}$.
\end{quote}
Then there is an element $z_{\Q}\in CH^i(X)\otimes
\Q_l$ such that
$$cl(z_{\Q})=\eta.$$
\end{cor}

\begin{proof}
The corollary follows from the theorem, the density theorem of
Chebotarev and the next lemma.
\end{proof}

\vspace{\abstand}

\begin{lemma}
Let $k$ be a field and $X$ a smooth, projective scheme over $k$. Let
$\eta\in{H^{2i}_\et(X_{\overline{k}},\Q_l(i))}$ be a cohomology class
which is invariant under the Galois group $Gal(\overline{k}/k)$. Assume
further that there is a field extension $K/k$, such that
$$\eta\in H^{2i}_\et(X_{\overline{k}},\Q_l(i))=H^{2i}_\et(X_{\overline{K}},\Q_l(i))$$
is in the image of
$$cl:CH^i(X_{K})\otimes \Q_l \rightarrow H^{2i}_\et(X_{\overline{K}},\Q_l(i)).$$
Then $\eta$ is in the image of
$$CH^i(X)\otimes \Q_l\rightarrow H^{2i}_\et(X_{\overline{k}},
\Q_l(i)).$$
\end{lemma}

\begin{proof}
By rigidity for Chow groups with finite coefficients, the map
$$ CH^i(X_{\overline{k}})/l^n\rightarrow CH^i(X_{\overline{K}})/l^n$$
is bijective. So we can assume that $K=\overline{k}$. Now let $\eta=cl(x)$
with $x\in CH^i(X_{\overline{k}})\otimes \Q_l$. Let $k'$ be a finite field extension of
$k$ with $k\subset k'\subset \overline{k}$, such that $x$ is already
defined over $k'$. By the Galois invariance of $\eta$ we have
$$cl(\sum_{\sigma\in Gal(k'/k)}\sigma\cdot x)= [k':k]\eta,$$
and the claim follows from Galois decent for Chow groups.
\end{proof}

The next theorem states that for divisors, we do not need the notion of complexity.

\begin{thm}\label{thm:withoutcompl}
Let $\fX\rightarrow \fB\subset\spec{\Z}$ be a smooth and proper
morphism with geometrically reduced fibers, $\cL$ an ample line
bundle on $\fX$, and let $\eta\in
H^2_\et(\fX\otimes_{\Q}\overline{\Q},\Z_l(i))$ be a cohomology class of
the generic fibre. We assume that for infinitely many closed points $s\in\fB$ there is
$z_s\in CH^1(\fX\otimes_{\Z}\kappa(s))$ such that
$$cl(z_s)=\eta.$$
Then there are a field extension $K/\Q$ and an element $z\in
CH^1(\fX_K)$ with
$$cl(z)=\eta \text{ in } H^2_\et(\fX_{\overline{K}},\Z_l(i)).$$
\end{thm}

\begin{proof}
Main parts of the proof are the same as in the proof of theorem \ref{firstmaintheorem}.

By transfer there is an $s\in\str{\fB}$, corresponding to an infinite
prime, such that there is a $z_s\in\str{CH}^1(\str{\fX}_{\kappa(s)})$
with
$$\str{cl}(z_s)=\str{\eta}.$$
Now the Hilbert polynomial of $z_s$ can be calculated in the \'{e}tale cohomology and is therefore
finite. Therefore the theorem follows from Corollary \cite{EoSII}[5.4]
\end{proof}


\bibliographystyle{alpha}
\bibliography{../Literatur}

\end{document}